\documentclass[12pt, reqno]{amsart}
\usepackage{amssymb,latexsym,amsmath,amscd,amsthm,graphicx, color}
\usepackage[all]{xy}
\usepackage{pgf,tikz}
\usepackage{mathrsfs}
\usetikzlibrary{arrows}
\usepackage[left=0.6 in, top=0.6 in, right=0.6 in, bottom=0.5 in]{geometry}

\makeatletter
\newcommand{\setword}[2]{%
  \phantomsection
  #1\def\@currentlabel{\unexpanded{#1}}\label{#2}%
}
\makeatother

\usepackage[linkcolor=blue, urlcolor=blue, citecolor=blue,
colorlinks, bookmarks]{hyperref}

\definecolor{uuuuuu}{rgb}{0.26666666666666666,0.26666666666666666,0.26666666666666666}
\definecolor{xdxdff}{rgb}{0.49019607843137253,0.49019607843137253,1.}
\definecolor{ffqqqq}{rgb}{1.,0.,0.}
\definecolor{ffqqqq}{rgb}{1.,0.,0.}
\definecolor{ffxfqq}{rgb}{1.,0.4980392156862745,0.}

\definecolor{uuuuuu}{rgb}{0.26666666666666666,0.26666666666666666,0.26666666666666666}
\definecolor{qqwuqq}{rgb}{0.,0.39215686274509803,0.}
\definecolor{zzttqq}{rgb}{0.6,0.2,0.}
\definecolor{xdxdff}{rgb}{0.49019607843137253,0.49019607843137253,1.}
\definecolor{qqqqff}{rgb}{0.,0.,1.}
\definecolor{cqcqcq}{rgb}{0.7529411764705882,0.7529411764705882,0.7529411764705882}
\definecolor{sqsqsq}{rgb}{0.12549019607843137,0.12549019607843137,0.12549019607843137}

\theoremstyle{plain}

\newtheorem{theorem}[subsection]{Theorem}
\newtheorem{theorem1}[subsubsection]{Theorem}

\newtheorem{lemma}[subsection]{Lemma}
\newtheorem{defi}[subsection]{Definition}
\newtheorem{prop}[subsection]{Proposition}

\newtheorem{lemma1}[subsubsection]{Lemma}

\theoremstyle{definition}

\newtheorem{remark}[subsection]{Remark}
\newtheorem{remark1}[subsubsection]{Remark}

\newtheorem{notation1}[subsubsection]{Notation}

\newcommand{\uu}{\cup}
\newcommand{\ii}{\cap}
\newcommand{\UU}{\bigcup}

\newcommand{\ci}{\subseteq}
\newcommand{\sci}{\subset}

\newcommand{\set}[1]{\{#1\}}

\newcommand{\ga}{\alpha}
\newcommand{\gb}{\beta}

\renewcommand{\gg}{\gamma}

\newcommand{\gk}{\kappa}
\newcommand{\gl}{\lambda}

\newcommand{\gq}{\theta}

\newcommand{\tit}{\textit}

\newcommand{\D}[1]{\mathbb{#1}}

\newcommand{\te}{\text}

\newcommand{\ol}{\overline}

\begin{document}
To appear, Journal `Mathematics'
\title{Conditional Optimal Sets and the Quantization Coefficients for Some Uniform Distributions}

\address{$^{2}$School of Mathematics \\
	Northwest University Xi'an\\
	Shaanxi Province, 710069, China.}
\address{$^{1, 3}$School of Mathematical and Statistical Sciences\\
University of Texas Rio Grande Valley\\
1201 West University Drive\\
Edinburg, TX 78539-2999, USA.}

\author{$^1$Evans Nyanney}
\author{$^2$Megha Pandey}
 \author{$^3$Mrinal Kanti Roychowdhury}

\email{\{$^1$evans.nyanney01, $^3$mrinal.roychowdhury\}@utrgv.edu}
\email{$^2$meghapandey1071996@gmail.com}

%


\date{}
\maketitle
 \pagestyle{myheadings}\markboth{E. Nyanney, M. Pandey, and M.K. Roychowdhury  }{Conditional Optimal Sets and the Quantization Coefficients for Some Uniform Distributions}

\begin{abstract}
Bucklew and Wise (1982) showed that the quantization dimension of an absolutely continuous probability measure on a given Euclidean space is constant and equals the Euclidean dimension of the space, and the quantization coefficient exists as a finite positive number.  By giving different examples, in this paper, we have shown that the quantization coefficients for absolutely continuous probability measures defined on the same Euclidean space can be different. We have taken uniform distribution as a prototype of an absolutely continuous probability measure. In addition, we have also calculated the conditional optimal sets of $n$-points and the $n$th conditional quantization errors for the uniform distributions in constrained and unconstrained scenarios. 
\end{abstract} 

\section{Introduction}

Quantization is the process of approximating a continuous-valued signal by a discrete set of values. It is a fundamental concept with widespread applications in engineering and technology. We refer to \cite{GG, GN, Z2} for surveys on the subject and comprehensive lists of references to the literature; see also \cite{AW, GKL, GL1, Z1}.

 \begin{defi}\label{EqVr1110}
Let $P$ be a Borel probability measure on $\D R^k$ equipped with a metric $d$ induced by a norm $\|\cdot\|$ on $\D R^k$. 
Let $S$ be a nonempty closed subset of $\D R^k$. Let $\gb\sci \D R^k$ be given with $\te{card}(\gb)=\ell$ for some $\ell\in \D N$. 
Then, for $n\in \D N$ with $n\geq \ell$, the \tit {$n$th conditional constrained quantization
error} for $P$ with respect to the constraint $S$ and the conditional set $\gb$, is defined by
\begin{equation}\label{EqVr121} 
V_{n}:=V_{n}(P)=\inf_{\ga} \Big\{\int \mathop{\min}\limits_{a\in\ga\uu\gb} d(x, a)^2 dP(x) : \ga \ci S, ~ 0\leq  \text{card}(\ga) \leq n-\ell \Big\},
\end{equation} 
where $\te{card}(A)$ represents the cardinality of the set $A$. 
\end{defi} 
 
\begin{defi}
A set $ \ga\uu\gb$, where $\ga \ci S$ and $P(M(b|\ga\uu \gb))>0$ for $b\in \gb$, for which the infimum in  \eqref{EqVr121} exists and contains no less than $\ell$ elements, and no more than $n$ elements is called a \tit{conditional constrained optimal set of $n$-points} for $P$ with respect to the constraint $S$ and the conditional set $\gb$.  
\end{defi}

We assume that $\int d(x, 0)^2 dP(x)<\infty$ to make sure that the infimum in \eqref{EqVr121} exists (see \cite{PR1}).  For a finite set $\gg \sci \D R^2$ and $a\in \gg$, by $M(a|\gg)$ we denote the set of all elements in $\D R^2$ which are nearest to $a$ among all the elements in $\gg$, i.e.,
$M(a|\gg)=\set{x \in \D R^2 : d(x, a)=\mathop{\min}\limits_{b \in \gg}d(x, b)}.$
$M(a|\gg)$ is called the \tit{Voronoi region} in $\D R^2$ generated by $a\in \gg$.

 Let $V_{n, r}(P)$ be a strictly decreasing sequence, and write $V_{\infty, r}(P):=\mathop{\lim}\limits_{n\to \infty} V_{n, r}(P)$. The numbers
\begin{equation} \label{eq55} \underline D(P):=\liminf_{n\to \infty}  \frac{2\log n}{-\log (V_n(P)-V_\infty(P))} \te{ and } \ol D(P):=\limsup_{n\to \infty} \frac{2\log n}{-\log (V_n(P)-V_\infty(P))}, \end{equation}
are called the \tit{conditional lower} and the \tit{conditional upper constrained quantization dimensions} of the probability measure $P$, respectively. If $\underline D (P)=\ol D (P)$, the common value is called the \tit{conditional constrained quantization dimension} of $P$ and is denoted by $D(P)$. For any $\gk>0$, the two numbers $\liminf_n n^{\frac 2 \gk}  (V_n(P)-V_\infty(P))$ and $\limsup_n  n^{\frac 2 \gk}(V_n(P)-V_\infty(P))$ are, respectively, called the \tit{$\gk$-dimensional conditional lower} and \tit{conditional upper constrained quantization coefficients} for $P$. If both of them are equal, then it is called the \tit{$\gk$-dimensional conditional constrained quantization coefficient} for $P$, and is denoted by $\lim_n  n^{\frac 2 \gk}(V_n(P)-V_\infty(P))$.   
\par
 If there is no conditional set, then by the \tit{$n$th constrained quantization error} for $P$ with respect to the constraint $S\ci \D R^k$, it is meant that 
\begin{equation} \label{eq001}
V_{n}:=V_{n}(P)= \inf\Big\{\int \mathop{\min}\limits_{a\in\ga} d(x, a)^2 dP(x) :  \ga\ci S \te{ and } 1\leq \text{card}(\ga) \leq n\Big\},
\end{equation}
and then the numbers $D(P)$ and $\lim_n  n^{\frac 2 \gk}(V_n(P)-V_\infty(P))$, if they exist, are called the \tit{constrained quantization dimension} and the \tit{$\gk$-dimensional constrained quantization coefficient} for $P$, respectively. A set $\ga\ci S$ for which the infimum in \eqref{eq001} exists is called a \tit{constrained optimal set of $n$-points} for $P$. 
\par 
If there is no constraint, i.e., if $S=\D R^k$, then by the \tit{$n$th conditional unconstrained quantization error} with respect to the conditional set $\gb$, it is meant that 
\begin{equation} \label{eq002}
V_{n}:=V_{n}(P)=\inf_{\ga} \Big\{\int \mathop{\min}\limits_{a\in\ga\uu\gb} d(x, a)^2 dP(x) : \ga \ci \D R^k, ~ 0\leq  \text{card}(\ga) \leq n-\ell \Big\},
\end{equation}
and then the numbers $D(P)$ and $\lim_n  n^{\frac 2 \gk}(V_n(P)-V_\infty(P))$, if they exist, are called the \tit{conditional unconstrained quantization dimension} and the \tit{$\gk$-dimensional conditional unconstrained quantization coefficient} for $P$, respectively. A set $\ga\ci S$ for which the infimum in \eqref{eq002} exists is called a \tit{conditional unconstrained optimal set of $n$-points} for $P$.
\par 
If there is no constraint and no conditional set, then by the \tit{$n$th unconditional quantization error} it is meant that 
\begin{equation} \label{eq003}
V_{n}:=V_{n}(P)=\inf \Big\{\int \mathop{\min}\limits_{a\in\ga} d(x, a)^2 dP(x) : \ga \ci \D R^k, ~ 1\leq  \text{card}(\ga) \leq n \Big\},
\end{equation}
and then the numbers $D(P)$ and $\lim_n  n^{\frac 2 \gk}(V_n(P)-V_\infty(P))$, if they exist, are called the \tit{unconstrained quantization dimension} and the \tit{$\gk$-dimensional unconstrained quantization coefficient} for $P$, respectively. A set $\ga\ci S$ for which the infimum in \eqref{eq003} exists is called an \tit{optimal set of $n$-means} for $P$. It is known that if the support of $P$ contains infinitely many elements than an optimal set of $n$-means contains exactly $n$ elements, and $V_\infty=\lim_{n\to \infty} V_n=0$.
\par
Constrained quantization and conditional quantization have recently been introduced by Pandey and Roychowdhury (see \cite{PR1, PR4}). After the introduction of constrained quantization, the quantization theory has now two classifications: constrained quantization and unconstrained quantization. Unconstrained quantization is traditionally known as quantization. Thus, the \tit{$n$th unconditional quantization error}, given by \eqref{eq003}, will traditionally be referred to as \tit{$n$th quantization error}. Likewise, \tit{unconstrained quantization dimension} and the \tit{$\gk$-dimensional unconstrained quantization coefficient} for $P$ will be referred to as \tit{quantization dimension} and the \tit{$\gk$-dimensional quantization coefficient} for $P$, respectively. For some other papers in the direction of constrained quantization and conditional quantization one can see \cite{BCDRV,BCDR,HNPR,PR3}. For unconstrained quantization, one can see \cite{ DFG, GG, GL2, GL3, GL, GN, GL1,  KNZ, P, P1, Z1, Z2} and the references therein. 
\par 

A seminal result by Bucklew and Wise (see \cite{BW}) established that for absolutely continuous probability measures defined on Euclidean spaces, the quantization dimension equals the Euclidean dimension of the space, and the quantization coefficient exists as a finite, positive constant.
This paper is motivated by the observation that although the quantization dimension remains invariant, the quantization coefficient may vary, even among absolutely continuous probability measures defined on the same Euclidean space. To investigate this phenomenon, we consider uniform distributions as prototypical examples of absolutely continuous measures.

The primary objectives of this paper are as follows:
\begin{itemize}
    \item To demonstrate, through various examples, that the quantization coefficients for absolutely continuous probability measures on the same Euclidean space can differ;
    \item To compute conditional optimal sets of $n$-points and the corresponding $n$th quantization errors under constrained and unconstrained scenarios;
    \item To examine the asymptotic quantization behavior in the presence and absence of conditional sets, particularly when the conditional set does not affect the quantization coefficient.
\end{itemize}

Additionally, the paper explores how the geometry of the support, such as line segments, circles, and boundaries of regular polygons, affects the quantization coefficient. This investigation highlights the structural dependencies of quantization efficiency, with a focus on determining how the shape and size of the support influence the values of the quantization coefficients.
 
\subsection*{Delineation} 
In Section~\ref{sec1} we give the basic preliminaries. In Section~\ref{sec2}, for a uniform distribution on a line segment taking different conditional sets we have calculated the conditional optimal sets of $n$-points and the $n$th conditional quantization errors. Then, for each conditional set we have calculated the quantization coefficient, and see that the quantization coefficient does not depend on the conditional set, but depends on the length of the line segment. In Section~\ref{sec3}, we have calculated the conditional optimal sets of $n$-points, the $n$th conditional quantization errors, the conditional quantization dimension, and the conditional quantization coefficient, in constrained scenario, for a uniform distribution defined on a circle of radius $r$ with respect to a given conditional set and a constraint. In addition, for the same probability distribution we have investigated the optimal sets of $n$-means and the $n$th quantization errors, and the quantization coefficient in unconstrained scenario. From the work in this section, we see that the quantization coefficient for a uniform distribution defined on a circle depends on the radius of the circle. In Section~\ref{sec4}, we have calculated the conditional optimal sets of $n$-points, the $n$th conditional quantization errors, and the conditional quantization coefficient for a uniform distribution defined on the boundary of a regular polygon which is inscribed in a circle of radius $r$ with respect to a given conditional set. From the work in this section, we see that the quantization coefficient for a uniform distribution defined on the boundary of a regular $m$-sided polygon depends on both the number of sides of the polygon and the length of the sides.

\section{Preliminaries} \label{sec1} 

For any two elements $(a, b)$ and $(c, d)$ in $\D R^2$, we write 
 \[\rho((a, b), (c, d)):=(a-c)^2 +(b-d)^2,\] which gives the squared Euclidean distance between the two elements $(a, b)$ and $(c, d)$. Let $p$ and $q$ be two elements that belong to an optimal set of $n$-points for some positive integer $n$, and let $e$ be an element on the boundary of the Voronoi regions of the elements $p$ and $q$. Since the boundary of the Voronoi regions of any two elements is the perpendicular bisector of the line segment joining the two elements, we have
\[\rho(p, e)-\rho(q, e)=0. \]
We call such an equation a \tit{canonical equation}. 

Let $P$ be a Borel probability measure on $\D R$ which is uniform on its support the closed interval $[a, b]$. Then, the probability density function $f$ for $P$ is given by 
\begin{align} \label{eq000} 
f(x)=\left\{\begin{array}{cc}
 \frac 1 {b-a} & \te{ if } a\leq x\leq b,\\
 0 & \te{ otherwise}.
\end{array}\right.
\end{align} 
Hence, we have $dP(x)=P(dx)=f(x) dx$ for any $x\in \D R$, where $d$ denotes the differential.  
   

 Let us now state the following proposition. For the details of the proof see \cite{BCDR}. 
\begin{prop}\emph{\cite{BCDR}} \label{prop0} 
Let $P$ be a uniform distribution on the closed interval $[a, b]$ and $c, d\in [a, b]$ be such that $a<c<d<b$. For $n\in \D N$ with $n\geq 2$, let $\ga_n$ be a conditional unconstrained optimal set of $n$-points for $P$ with respect to the conditional set $\gb=\set{c, d}$ such that $\ga_n$ contains $k$ elements from the closed interval $[a, c]$, $\ell$ elements from the closed interval $[c, d],$ and $m$ elements from the closed interval $[d, b]$ for some $k,\ell,  m\in \D N$ with $k, m\geq 1$ and $\ell\geq 2$. Then, $k+\ell+m=n+2$,
\begin{align*} \ga_n\ii [a, c]&= \Big\{a+\frac{(2j-1)(c-a)}{2k-1} : 1\leq j\leq k\Big\}, \\
 \ga_n\ii[c, d]&=\Big\{c+\frac{j-1}{\ell-1}(d-c) : 1\leq j\leq \ell\Big\},  \te{ and } \\
 \ga_n\ii [d, b]&=\Big\{d+\frac{2(j-1)(b-d)}{2m-1} : 1\leq j\leq m\Big\} 
 \end{align*} 
with the conditional unconstrained quantization error 
 \[V_n:=V_{k,\ell, m}(P)= \frac  1 {3(b-a)}\Big(\frac{(c-a)^3}{(2k-1)^2}+\frac 1 {4}\frac {(d-c)^3}{ (\ell-1)^2}+\frac{(b-d)^3}{(2m-1)^2}\Big).\]
 \end{prop}
 
 \begin{remark}
If nothing is specified, by conditional optimal sets of $n$-points and the conditional quantization errors, it is meant the conditional unconstrained optimal sets of $n$-points and the conditional unconstrained quantization errors, respectively. 
 \end{remark} 
The following theorem motivates us to make Remark~\ref{rem3}, Remark~\ref{rem5}, and Remark~\ref{rem8}. 
 \begin{theorem} \emph{(see \cite{PR4})} \label{TheoM1} 
In both constrained and unconstrained quantization, the lower and upper quantization dimensions and the lower and upper quantization coefficients for a Borel probability measure do not depend on the conditional set. 
\end{theorem}

\begin{remark}
Given that the underlying spaces for all considered probability measures $P$ in this work are one dimensional, their quantization dimensions are given by  $D(P)=1$ (see \cite{BW}).
Hence, in the sequel we are mostly interested to calculate the quantization coefficients for different uniform distributions, though in some cases we have also calculated the conditional optimal sets of $n$-points and the $n$th conditional quantization errors in constrained and unconstrained scenarios. 
\end{remark}

In the following sections we give the main results of the paper.

 \section{Conditional Optimal Sets of $n$-Points and the Quantization Coefficients for Uniform Distributions on Line Segments} \label{sec2} 
 
Without any loss of generality we can assume the line segment as a closed interval $[a, b]$, where $0<a<b<+\infty$. Let $P$ be the uniform distribution defined on the closed interval $[a, b]$. Then, the probability density function $f$ for $P$ is given by \eqref{eq000}.  

The following theorem is a consequence of Proposition~\ref{prop0}, where $c = a$ and $d = b$; $c = d = a$; and $c = d = b$, respectively.
\begin{theorem}\label{theoM0}
Let $P$ be the uniform distribution on the line segment joining $a$ and $b$, where $a, b\in \D R$ with $a<b$. Then we have:  

$(i)$ the conditional optimal set of $n$-points with respect to the conditional set $\gb:=\set{a, b}$ is  
\[\Big\{a+\frac{j-1}{\ell-1}(b-a) : 1\leq j\leq n\Big\} \te{ with conditional quantization error } V_n=\frac {(b-a)^2}{12(n-1)^2};\]

$(ii)$ the conditional optimal set of $n$-points with respect to the conditional set $\gb:=\set{a}$ is  
\[\Big\{a+\frac{2(j-1)(b-a)}{2n-1}   : 1\leq j\leq n\Big\} \te{ with conditional quantization error } V_n=\frac {(b-a)^2}{3(2n-1)^2};\]

$(iii)$ the conditional optimal set of $n$-points with respect to the conditional set $\gb:=\set{b}$ is  
\[\Big\{a+\frac{(2j-1)(b-a)}{2n-1}   : 1\leq j\leq n\Big\} \te{ with conditional quantization error } V_n=\frac {(b-a)^2}{3(2n-1)^2}.\]
\end{theorem}


\begin{theorem}\label{theoM00}
Let $P$ be the uniform distribution on the line segment joining $a$ and $b$, where $a, b\in \D R$ with $a<b$. Then, the conditional quantization coefficients for $P$  with respect to the conditional sets $\set{a, b}$, $\set{a}$, and $\set{b}$ exist as finite positive numbers and each equals $\frac{(a-b)^2}{12}$. 
\end{theorem}

\begin{proof}
By Theorem~\ref{theoM0}$(i)$, we obtain the $n$th conditional quantization error for the uniform distribution $P$ with respect to the conditional set $\gb:=\set{a, b}$ as 
\[V_n=\frac {(b-a)^2}{12(n-1)^2}.\]
Then, \[V_\infty=\mathop{\lim}\limits_{n\to\infty} V_n=0 \te{ yielding } \lim\limits_{n\to \infty} n^2 (V_n-V_\infty)=\frac{(a-b)^2}{12}.\]
Similarly, if the conditional set is $\gb:=\set{a}$ or $\gb:=\set{b}$, by $(ii)$ and $(iii$) in Theorem~\ref{theoM0}, we obtain $\lim\limits_{n\to \infty} n^2 (V_n-V_\infty)=\frac{(a-b)^2}{12}$. 
\end{proof} 

\begin{remark} \label{rem3}
By Theorem~\ref{TheoM1} and Theorem~\ref{theoM00}, we see that the quantization coefficient for the uniform distribution $P$ on a line segment depends on the length of the line segment, and does not depend on the conditional sets. 
\end{remark}

 
\section{Optimal sets of $n$-Points and Quantization Coefficients for Uniform Distributions on Circles: Conditional Constrained, and Unconstrained Scenarios} \label{sec3} 
  
In this section, we have two subsections. In the first subsection,  we calculate the conditional constrained optimal sets of $n$-points and the $n$th conditional constrained quantization errors, the conditional constrained quantization dimension, and the conditional constrained quantization coefficient for a uniform distribution $P$ defined on a circle of radius $r$ with respect to a given conditional set and a constraint. In the second subsection, for the same probability distribution, we investigate the optimal sets of $n$-means and the $n$th quantization errors, and the quantization coefficient in unconstrained scenario.  
\par
Let $L$ be the circle of radius $r$. Without any loss of generality, we can take the equation of the circle as $x_1^2+x_2^2=r^2$, i.e., the parametric equations of the circle is given by
$L:=\set{(x_1, x_2) :  x_1=r\cos \gq, \, x_2=r\sin\gq \te{ for } 0\leq \gq\leq 2\pi}.$ Notice that any point on the circle can be given by $(r\cos\gq, r\sin \gq)$, which will be identified as $\gq$, where $0\leq \gq\leq 2\pi$. 
Let the positive direction of the $x_1$-axis cut the circle at the point $A$, i.e., $A$ is represented by the parametric value $\gq=0$. Let $s$ be the distance of a point on $L$ along the arc starting from the point $A$ in the counterclockwise direction. Then,
\[ds=\sqrt{\Big(\frac {dx_1}{d\gq}\Big)^2+\Big(\frac{dx_2}{d\gq}\Big)^2}\,d\gq=rd\gq,\]
where $d$ stands for differential. 
 Then, the  probability density function (pdf) $f(x_1, x_2)$ for $P$ is given by
\[f(x_1, x_2)=\left\{\begin{array}{ccc}
\frac 1 {2\pi r} & \te{ if } (x_1, x_2) \in L,\\
 0  & \te{ otherwise}.
\end{array}\right.
\]
Thus, we have $dP(s)=P(ds)=f(x_1, x_2) ds=\frac 1{2\pi } d\gq$.
Moreover, we know that if $\hat \gq$ radians is the
central angle subtended by an arc of length $S$ of the circle, then $S =r \hat \gq$, and
\[P(S)=\int_S dP(s)=\frac 1{2 \pi} \int_S d\gq= \frac{\hat\gq} {2\pi}.\]

\subsection{Conditional Quantization in Constrained Scenario} \label{sub1}
In this subsection to investigate the conditional quantization in constrained scenario for the uniform distribution $P$ on the circle $L$, we take the circle $L$ as the constraint and the set $\set{(r, 0)}$ as the conditional set. 
Let us define a function 
\[T : L \to [0, 2\pi r]  \te{ such that } T(\gq):=T((r\cos \gq, r\sin \gq))=r \gq,\] where $0\leq \gq\leq 2 \pi$. Then, notice that $T : L\setminus \set{(r, 0)} \to (0, 2\pi r) $ is a bijective function. Let $Q$ be the image measure of $P$ under the function $T$, i.e., $Q=TP$ such that for any Borel subset $A \ci [0, 2\pi r]$, we have 
\[Q(A)=TP(A)=P(T^{-1}(A)).\]
\begin{lemma1} \label{lemma21} 
The image measure $Q$ is a uniform distribution on $[0, 2\pi r]$. 
\end{lemma1} 
\begin{proof}
Since $P$ is a uniform distribution on $L$, we can assume that $P$ is also a uniform distribution on $L\setminus \set{(r, 0)}$, as the deletion, or addition, of a finite number of points from, or with, the support of a continuous probability measure does not change the distribution. 
Take any $[c, d]\ci (0, 2\pi r)$, where $0< c<d<2\pi r$. Since $T$ is a bijection, there exist $\gq_1$, $\gq_2$, where $0< \gq_1<\gq_2<2\pi $, such that $T(\gq_1)=r\gq_1=c$ and $T(\gq_2)=r\gq_2=d$. Then,   
\begin{align*}
Q([c, d])&=P(T^{-1}([c, d]))=P(\set{(r\cos \gq, r\sin \gq) : \gq_1\leq \gq\leq \gq_2})=\frac{\gq_2-\gq_1}{2\pi}\\
&=\frac{r\gq_2-r\gq_1}{2\pi r}=\frac {d-c}{2\pi r}.
\end{align*}
Notice $Q([c, d])=\gl([c, d])$, where $\gl$ is the normalized Lebesgue measure on $(0, 2\pi r)$. Hence, we can conclude that $Q$ is a uniform distribution on $(0, 2\pi r$), i.e., $Q$ is a uniform distribution on $[0, 2\pi r]$. 
\end{proof}

\begin{notation1}
For any two elements $c, d\in [0, 2\pi r]$, by $\ell(\set{c, d})$ it is meant the distance between the two elements $c, d$, i.e., $\ell(\set{c, d})=d-c$. Similarly, for any two elements $\gq_1, \gq_2 \in L$ with $\gq_1<\gq_2$, by $\ell(\set{\gq_1, \gq_2})$ it is meant the arc distance between the two elements $\gq_1$ and $\gq_2$, i.e., the length of the arc on $L$ subtended by the angle $\gq_2-\gq_1$, i.e.,  $\ell(\set{\gq_1, \gq_2})=r(\gq_2-\gq_1).$
\end{notation1} 

\begin{lemma1}\label{lemma22} 
The function $T: L\setminus \set{(r, 0)} \to (0, 2\pi r) $ preserves the distance. 
\end{lemma1} 
\begin{proof} 
Take any $c, d\in (0, 2\pi r)$ such that $0< c<d<2\pi r$. The lemma will be proved if we can prove that $\ell(\set{c, d})=\ell(T^{-1}(\set{c, d}))$. Since $T: L\setminus \set{(r, 0)} \to (0, 2\pi r) $ is a bijection, there exist $\gq_1, \gq_2\in L\setminus \set{(r, 0)}$ such that $T(\gq_1)=r\gq_1=c$ and $T(\gq_2)=r\gq_2=d$. Then, 
\[\ell(T^{-1}(\set{c, d}))=\ell(\set{\gq_1, \gq_2})=r(\gq_2-\gq_1)=d-c =\ell(\set{c, d}).\]
Thus, the lemma is yielded. 
\end{proof} 
The following lemma is a consequence of Proposition~\ref{prop0}, where $a=c=0$ and $b=d=2\pi r$.
\begin{lemma1}\label{lemma23} 
The conditional unconstrained optimal set for the uniform distribution $Q$ with respect to the conditional set $\set{0, 2\pi r}$ is given by 
$\Big\{\frac{(j-1)2\pi r}{n-1} : 1\leq j\leq n\Big\}$ 
with  conditional unconstrained quantization error $V_n(Q)=\frac{\pi^2 r^2}{3(n-1)^2}.$
\end{lemma1}

\begin{remark1} \label{rem0} 
Let $\set{a_1, a_2, \cdots, a_n}$ be a conditional unconstrained optimal set of $n$-points for $Q$ with respect to the conditional set $\set{0, 2\pi r}$, where $a_1=0$ and $a_n=2\pi r$,
Since both $P$ and $Q$ are uniform distributions, and $Q$ is the image measure of $P$ under the function $T$, and $T: L\setminus \set{(r, 0)} \to (0, 2\pi r) $ preserves the distance, we can say that the set $T^{-1}(\set{a_1, a_2, \cdots, a_{n}})$, i.e., the set  $\set{T^{-1}(a_j) : 1\leq j\leq n-1}$ forms a conditional constrained optimal set of $(n-1)$-points for $P$ with respect to the conditional set $\set{(r, 0)}$ and the constraint $L$, as $T^{-1}(0)=T^{-1}(2\pi r)= (r, 0) $. 
\end{remark1}

Let us now prove the following theorems, which give the main results in this subsection. 

\begin{theorem1} \label{theoM001}
Let $P$ be the uniform distribution on the circle of radius $r$ with center $(0, 0)$. Then, the set $\set{(r\cos \frac{(j-1)2\pi}{n}, r\sin \frac{(j-1)2\pi}{n}) : 1\leq j\leq n}$ forms a conditional constrained optimal set of $n$-points with respect to the conditional set $\set{(r, 0)}$ and the constraint $L$ with conditional constrained quantization error 
\[V_n=2r^2(1-\frac n \pi\sin \frac \pi n).\]
\end{theorem1}

\begin{proof}
By Lemma~\ref{lemma23}, we know that the set $\{\frac{(j-1)2\pi r}{n-1} : 1\leq j\leq n\}$ forms a conditional unconstrained optimal set of $n$-points for $Q$ with respect to the conditional set $\set{0, 2\pi r}$, where $n\geq 2$. Hence, by Remark~\ref{rem0}, the set 
$\set{T^{-1}(\frac{(j-1)2\pi r}{n-1}) : 1\leq j\leq n-1}$ forms a conditional constrained optimal set of $(n-1)$-points for $P$ with respect to the conditional set $\set{(0, \pi)}$ and the constraint $L$, where $(n-1)\geq 1$.  
Now, notice that 
\begin{align*}
&\set{T^{-1}(\frac{(j-1)2\pi r}{n-1}) : 1\leq j\leq n-1}=\set{\frac{(j-1)2\pi}{n-1} : 1\leq j\leq n-1}\\
&=\set{(\cos \frac{(j-1)2\pi}{n-1}, \sin \frac{(j-1)2\pi}{n-1}) : 1\leq j\leq n-1}.
\end{align*}
Hence, replacing $n$ by $n+1$, we deduce that the set $\set{(r\cos \frac{(j-1)2\pi}{n}, r\sin \frac{(j-1)2\pi}{n}) : 1\leq j\leq n}$ forms a conditional constrained optimal set of $n$-points with respect to the conditional set $\set{(0, \pi)}$ and the constraint $L$. Due to rotational symmetry, we obtain the conditional constrained quantization error as 
\begin{align*}
V_n&=n  (\te{distortion error contributed by the element } (r, 0))\\
&=\frac {n} {2\pi} \int_{-\frac {\pi} n}^{\frac {\pi} n} \rho((r\cos \gq, r\sin \gq), (r, 0))\,d\gq\\
&=2r^2(1-\frac n \pi\sin \frac \pi n). 
\end{align*} 
\end{proof}

\begin{theorem1}\label{theo451}
Let $P$ be the uniform distribution on the circle of radius $r$ with center $(0, 0)$. Then, with respect to the conditional set $\set{(r, 0)}$ and the constraint $L$, the conditional constrained quantization dimension $D(P)$ exists and equals one, and the conditional constrained quantization coefficient for $P$ exists as a finite positive number and equals $\frac{\pi ^2 r^2}{3}$, i.e.,
$\lim\limits_{n\to \infty} n^2 (V_n-V_\infty)=\frac{\pi ^2 r^2}{3}.$
\end{theorem1}

\begin{proof}
By Theorem~\ref{theoM001}, we obtain the $n$th conditional constrained quantization error for the uniform distribution $P$ with respect to the conditional set $\set{(r, 0)}$ and the constraint $L$ as 
$V_n=2r^2(1-\frac n \pi\sin \frac \pi n).$
Then, $V_\infty=\mathop{\lim}\limits_{n\to\infty} V_n=0$. 
Hence, 
\[D(P)=\lim_{n\to \infty} \frac{2\log n}{-\log (V_n-V_\infty)}=1 \te{ and } \lim\limits_{n\to \infty} n^2 (V_n-V_\infty)=\frac{\pi ^2 r^2}{3}.\]
\end{proof} 
\begin{remark1}\label{rem5} 
By Theorem~\ref{TheoM1}, we know that quantization dimension and the quantization coefficient in constrained and unconstrained cases do not depend on the conditional set. Thus, 
by Theorem~\ref{theo451}, we see that constrained quantization dimension of the uniform distribution $P$ with respect to the constraint $L$ equals one, which is the dimension of the underlying space where the support of the probability measure is defined, and does not depend on the radius $r$ of the circle. This fact is not true, in general, in constrained quantization, for example, one can see \cite{PR3, PR1}. However, we see that the constrained quantization coefficient for the uniform distribution $P$ with respect to the constraint $L$ depends on the radius $r$ of the circle. 
\end{remark1}

 \subsection{Quantization in Unconstrained Scenario}
In this subsection, we investigate the optimal sets of $n$-means, $n$th quantization errors, and the quantization coefficient for the uniform distribution $P$ when there is no constraint and no conditional set. The following lemma is a generalized version of a similar theorem that appears in \cite{RR}.   
\begin{theorem1} \label{theo452}
Let $\ga_n$ be an optimal set of $n$-means for the uniform distribution $P$ on the circle $x_1^2+x_2^2=r^2$ for $n\in\D N$. Then,
\[\ga_n:=\Big\{ \frac {nr}{ \pi} \Big(\sin (\frac \pi n) \cdot \cos ((2j-1){\frac \pi n}), \   \sin (\frac \pi n) \cdot \sin ((2j-1){\frac \pi n})\Big) : j=1, 2, \cdots, n \Big  \} \]
forms an optimal set of $n$-means, and the corresponding quantization error is given by $r^2(1-\frac{n^2}{\pi^2}  \sin^2\frac{\pi }{n}).$
\end{theorem1}
\begin{proof}
Let $\ga_n:=\set{a_1, a_2, \cdots, a_n}$ be an optimal set of $n$-means for $P$. Let the boundaries of the Voronoi regions of $a_k$ intersect $L$ at the points given by the parameters $\gq_{k-1}$ and $\gq_k$ such that $\gq_{k-1}<\gq_k$ for $1\leq k\leq n$. Without any loss of generality, we can assume that $\gq_0=0$ and $\gq_n=2\pi$. Since, $P$ is a uniform distribution and the circle $L$ is rotationally symmetric, without going into the much details of calculations, we see that  
\begin{equation*} \gq_1-\gq_0=\gq_2-\gq_1=\gq_3-\gq_2=\cdots=\gq_n-\gq_{n-1}=\frac{\gq_n-\gq_0}{n}=\frac{2\pi}{n}.
\end{equation*}
implying $\gq_k=\frac {2\pi k}{n}$ for $0\leq k\leq n$. 
It is well-known that in unconstrained quantization, the elements in an optimal set are the conditional expectations in their own Voronoi regions.  Hence, for $1\leq k\leq n$, we have 
 \begin{align*}
a_k&=\frac {2 \pi}  {\gq_{k}-\gq_{k-1}} \int_{\gq_{k-1}}^{\gq_k} \frac 1 {2\pi}( r\cos \gq, r\sin \gq) d\gq=\frac r {\gq_k-\gq_{k-1}} \left(\sin \gq_{k}-\sin\gq_{k-1}, \, \cos\gq_{k-1}-\cos \gq_k\right), 
\end{align*}
yielding \[ a_k=\frac {nr}{ \pi} \Big(\sin (\frac \pi n) \cdot \cos ((2k-1){\frac \pi n}), \   \sin (\frac \pi n) \cdot \sin ((2k-1){\frac \pi n})\Big).\]
Let $V_n$ be the $n$th quantization error. Due to symmetry, the distortion errors contributed by $a_k$ in their own Voronoi regions are equal for all $1\leq k\leq n$. Again, notice that $\gq_0=0$, $\gq_1=\frac{2\pi}{n}-\gq_0=\frac {2\pi} n$, and $a_1=\frac r{\gq_1}(\sin \gq_1, 1-\cos \gq_1)$. Hence, 
\begin{align*}
V_n&= \frac n{2\pi}\int_{\gq_0}^{\gq_1} \rho((r\cos \gq, r\sin \gq), a_1) \, d \gq =r^2(1-\frac{n^2}{\pi^2}  \sin^2\frac{\pi }{n}).
\end{align*} 
\end{proof}

By Theorem~\ref{theo452}, we have
$
V_n = r^2 \left(1 - \frac{n^2}{\pi^2} \sin^2 \frac{\pi}{n} \right),
$
and hence
$
\lim_{n \to \infty} n^2 V_n = \frac{\pi^2 r^2}{3},
$
which motivates us to give the following theorem.

\begin{theorem1}\label{theo453}
Let $P$ be the uniform distribution on the circle of radius $r$ with center $(0, 0)$. Then, the quantization coefficient for $P$ exists as a finite positive number and equals $\frac{\pi ^2 r^2}{3}$.
\end{theorem1}   
 
 \begin{remark1}\label{rem6}
 Theorem~\ref{theo453} implies that the quantization coefficient for a uniform distribution on a circle of radius $r$, though exists as a finite positive number, depends on the radius $r$ of the circle. 
 \end{remark1}

 \section{Conditional Optimal Sets and the Quantization Coefficients for the Uniform Distributions on the Boundaries of the Regular Polygons} \label{sec4} 
In this section, for a uniform distribution defined on the boundary of a regular $m$-sided polygon, we calculate the conditional optimal sets of $n$-points and the $n$th conditional quantization errors, and the conditional quantization coefficient taking the conditional set as the set of all vertices of the polygon. \par
Let $P$ be the uniform distribution defined on the boundary $L$ of a regular $m$-sided polygon given by $A_1A_2\cdots A_m$ for some $m\geq 3$. Without any loss of generality we can assume that the polygon is inscribed in the circle $x^2+y^2=r^2$ which has center $O(0, 0)$ and radius $r$ with the Cartesian coordinates of the vertex $A_1$ as $(r, 0)$. Let $\gq$ be the central angle subtended by each side of the polygon, and let $\gq_j$ be the polar angles of the vertices $A_j$. Then, we have $\gq=\frac {2\pi} m$ and $\gq_j=(j-1)\frac {2\pi} m$. Then, the polar coordinates of the vertices $A_j$ are given by $(r\cos \gq_j, r\sin \gq_j)$. Hence, if $\ell$ is the length of each of the sides $A_jA_{j+1}$ for $1\leq j\leq m$, where the vertex $A_{m+1}$ is identified as the vertex $A_1$, then we have  
\[\ell= \te{length of } A_j A_{j+1}=\sqrt{\rho((r\cos \gq_j, r\sin \gq_j), (r\cos \gq_{j+1}, r\sin \gq_{j+1}))}=2r\sin \frac{\pi}m.\]
The probability density function (pdf) $f$ for the uniform distribution $P$ is given by $f(x, y)=\frac 1{m\ell}$ for all $(x, y)\in A_1A_2\cdots A_m$, and zero otherwise. Moreover, we can write 
\[L=\UU_{j=1}^m L_j, \te{ where } L_j \te{ represents the side } A_jA_{j+1} \te{ for } 1\leq j\leq m.\]
Notice that 
\[A_jA_{j+1}=\set{(1-t)(r \cos \gq_j, r\sin \gq_j)+t(r\cos \gq_{j+1}, r\sin \gq_{j+1}) : 0\leq t\leq 1}.\] 
for $1\leq j\leq m$. Write 
\begin{align} a_j:&=-\sec \frac{\pi }{m} \Big(j \sin\frac{2 \pi  (j-1)}{m}-(j-1) \sin \frac{2 \pi  j}{m}\Big), \label{Me1}\\
b_j:&=-2 \sin \frac{\pi }{m} \csc \frac{2 \pi }{m} \Big((j-1) \cos \frac{2 \pi  j}{m}-j \cos\frac{2 \pi  (j-1)}{m}\Big), \label{Me2}\\
c_j:&=2 (j-1) r \sin \frac{\pi }{m}.\label{Me3}
\end{align} 
Let us consider the affine transformation  
\[T: L \to [0, 2mr\sin \frac{\pi} m]  \te{ such that } T(x, y)=a_j x +b_j y \te{ if }  (x, y) \in A_jA_{j+1},\]
where $a_j$ and $b_j$ are given by \eqref{Me1} and \eqref{Me2} for all  $1\leq j\leq m$. 
 Then, notice that $T : L\setminus \set{(r, 0)} \to (0, 2mr\sin \frac{\pi} m) $ is a bijective function. 
   Let $Q$ be the image measure of $P$ under the function $T$, i.e., $Q=TP$ such that for any Borel subset $A \ci [0, 2mr\sin \frac{\pi} m]$, we have 
\[Q(A)=TP(A)=P(T^{-1}(A)).\]
By the distance between any two elements in $[0, 2mr\sin \frac{\pi} m]$, it is meant the Euclidean distance between the two elements. On the other hand, by the distance between any two elements on $L$, it is meant the Euclidean distance between the two elements along the polygonal arc $L$ in the counterclockwise direction. Let $T_j$ be the restriction of the mapping $T$ to the set $A_jA_{j+1}$, i.e., $T_j=T|_{A_jA_{j+1}}$ for $1\leq j\leq m$. Notice that each $T_j$ is a bijective function.  
 \begin{lemma}\label{lemma221} 
The function $T: L\setminus \set{(r, 0)} \to (0,2mr\sin \frac{\pi} m) $ preserves the distance. 
\end{lemma} 
\begin{proof} 
Notice that 
\begin{align*} [0, 2mr\sin \frac{\pi} m]=\UU_{j=1}^m [c_j, c_{j+1}]  \te{ and } T(A_jA_{j+1})=T_j(A_jA_{j+1})=[c_j, c_{j+1}],
\end{align*}
where $c_j$ are given by \eqref{Me3} for all $1\leq j\leq m$. 
Since the length of $A_jA_{j+1}$ equals the length of the closed interval $[c_j, c_{j+1}]$, and $T_j$ is a bijection, we can say that $T: L\setminus \set{(r, 0)} \to (0,2mr\sin \frac{\pi} m) $ preserves the distance. 
Thus, the lemma is yielded. 
\end{proof} 

The following lemma which is similar to Lemma~\ref{lemma21} is also true here. 
\begin{lemma} \label{lemma211} 
The image measure $Q$ is a uniform distribution on $[0,2m r\sin \frac{\pi} m]$. 
\end{lemma} 

\begin{remark} \label{rem0} 
Let $\set{a_1, a_2, \cdots, a_n}$ be a conditional unconstrained optimal set of $n$-points for $Q$ with respect to the conditional set $\set{c_j : 1\leq j \leq m+1}$ such that $a_1<a_2<\cdots <a_n$. Then, by the definition of conditional set, we have $n\geq m+1$. Moreover, notice that $a_1=c_1=0$ and $a_n=c_{m+1}=2m r\sin \frac{\pi} m$. 
Since both $P$ and $Q$ are uniform distributions, and $Q$ is the image measure of $P$ under the function $T$, and $T: L\setminus \set{(r, 0)} \to (0, 2m r\sin \frac{\pi} m) $ preserves the distance, we can say that the set $T^{-1}(\set{a_1, a_2, \cdots, a_{n}})$, i.e., the set  $\set{T^{-1}(a_j) : 1\leq j\leq n-1}$ forms a conditional unconstrained optimal set of $(n-1)$-points for $P$ with respect to the conditional set $\set{T^{-1}(c_j) : 1\leq j\leq m}$, i.e., with respect to the conditional set $\set{A_j : 1\leq j\leq m}$. 
\end{remark}

\begin{lemma} \label{lemma34} 
Let $\gg_n$ be a conditional optimal set of $n$-points for the uniform distribution $Q$ with respect to the conditional set $\set{c_j : 1\leq j\leq m+1}$ for any $n\geq m+1$. Let $n_j=\te{card}(\gg_n\ii [c_j, c_{j+1}])$ for $1\leq j\leq m$. Then, $n_j\geq 2$ and $ n_1+n_2+\cdots +n_m=n+m-1$, and $|n_i-n_{j}|=0 \te{ or } 1$  for all $1\leq i\neq  j\leq m$. 
\end{lemma} 
\begin{proof}
Let $\gg_n$ be a conditional optimal set of $n$-points and $n_j$ be the positive integers as defined in the hypothesis. Notice that each of the sets $\gg_n\ii [c_j, c_{j+1}]$ always contains the end elements $c_j,\, c_{j+1}$ for $1\leq j\leq m$, where $n\geq m+1$. Moreover, except the two elements $c_1$ and  $c_{m+1}$ all the end elements $c_j$ are counted two times. Hence,  $n_j\geq 2$ and $n_1+n_2+\cdots +n_m=n+(m+1-2)=n+m-1$.  
Since $Q$ is a unform distribution and the lengths of the intervals $[c_j, c_{j+1}]$ for $1\leq j\leq m$ are all equal, the proof of $|n_i-n_{j}|=0 \te{ or } 1$  for all $1\leq i\neq  j\leq m$ is routine. 
\end{proof} 
 
Let us now give the following proposition. 

\begin{prop} \label{prop45} 
Let $\gg_n$ be a conditional optimal set of $n$-points and $V_n(Q)$ be the $n$th conditional quantization error for the uniform distribution $Q$ with respect to the conditional set $\set{c_j : 1\leq j\leq m+1}$ for any $n\geq m+1$. Let $n=m k +1+q$, where $0\leq q< m$. Then, if $q=0$,  
we have 
\begin{align} \label{Megha101} 
\gg_n& = \UU_{j=1}^{m}\{c_j+\frac{2(i-1)r\sin \frac{\pi}m}{k}  : 1\leq i\leq k+1\} \te{ with } V_n(Q)=\frac {r^2 \sin^2\frac{\pi} m}{3k^2}.
\end{align} 
On the other hand, if $0<q<m$, then there are $^mC_q$ possible sets $\gg_n$, one such set is given by 
\begin{equation}\label{Megha102} 
\begin{aligned} 
\gg_n  &= \Big(\UU_{j=1}^{q}\{c_j+\frac{2(i-1)r\sin \frac{\pi}m}{k+1} : 1\leq i\leq k+2\}\Big)\\
&\qquad \UU \Big(\UU_{j=q+1}^{m}\{c_j+\frac{2(i-1)r\sin \frac{\pi}m}{k}  : 1\leq i\leq k+1\}\Big)\\
\te{ with }    V_n(Q)&=\frac {r^2q\sin^2\frac{\pi} m}{3m(k+1)^2}+ \frac {r^2(m-q)\sin^2\frac{\pi} m}{3mk^2}.
 \end{aligned}
\end{equation}  
 \end{prop} 

\begin{proof} 
First assume that $q=0$, then we have $n=mk+1$, i.e., $\gg_n$ contains $k-1$ elements from each of the intervals $[c_j, c_{j+1}]$ except the boundary elements $c_j$ and $c_{j+1}$. Hence, if $n_j=\te{card}(\gg_n\ii [c_j, c_{j+1}])$, we have $n_j=k-1+2=k+1$. Hence, by $(i)$ of Theorem~\ref{theoM0}, we have 
\begin{align*}
\gg_n& = \UU_{j=1}^{m}\{c_j+\frac{i-1}{k}(c_{j+1}-c_{j}) : 1\leq i\leq k+1\}=\UU_{j=1}^{m}\{c_j+\frac{2(i-1)r\sin \frac{\pi}m}{k}  : 1\leq i\leq k+1\}
\end{align*}
\[ \te{ with } V_n=\sum_{j=1}^m \frac {(c_{j+1}-c_j)^2}{12mk^2}=\sum_{j=1}^m \frac {r^2 \sin^2\frac{\pi} m}{3mk^2}=\frac {r^2 \sin^2\frac{\pi} m}{3k^2}.\]
On the other hand, if $0<q<m$, then due to Lemma~\ref{lemma34}, we can assume that $\gg_n$ contains $k$ elements from each of the first $q$ intervals except the boundary elements, and $\gg_n$ contains $(k-1)$ elements from each of the remaining $m-q$ intervals except the boundary elements implying $n_1=n_2=\cdots=n_q=k+2$ and $n_{q+1}=n_{q+2}=\cdots=n_m=k+1$. Hence, the expressions for $\gg_n$ and the corresponding $n$th conditional quantization errors  are  obtained by  $(i)$ of Theorem~\ref{theoM0} as 
\begin{align*}
\gg_n  &= \Big(\UU_{j=1}^{q}\{c_j+\frac{2(i-1)r\sin \frac{\pi}m}{k+1} : 1\leq i\leq k+2\}\Big)\\
&\qquad \UU \Big(\UU_{j=q+1}^{m}\{c_j+\frac{2(i-1)r\sin \frac{\pi}m}{k}  : 1\leq i\leq k+1\}\Big)
\end{align*} 
 with  
 \[V_n=\sum_{j=1}^q \frac {(c_{j+1}-c_j)^2}{12 m(k+1)^2}+\sum_{j=q+1}^m \frac {(c_{j+1}-c_j)^2}{12mk^2}=\sum_{j=1}^q \frac {r^2\sin^2\frac{\pi} m}{3m(k+1)^2}+\sum_{j=q+1}^m \frac {r^2\sin^2\frac{\pi} m}{3mk^2}\] 
 yielding \[V_n=\frac {r^2q\sin^2\frac{\pi} m}{3m(k+1)^2}+ \frac {r^2(m-q)\sin^2\frac{\pi} m}{3mk^2}.\]
 Notice that if $n=mk+q$, the optimal set $\gg_n$ can be constructed in $^mC_q$ ways. 
\end{proof} 

The following two theorems give the main results in this section. 

\begin{theorem} \label{theMe0}
Let $\ga_n$ be a conditional optimal set of $n$-points and $V_n(P)$ be the $n$th conditional quantization error for the uniform distribution $P$ with respect to the conditional set $\set{A_j : 1\leq j\leq m}$ for any $n\geq m$. Let $n=m k +q$, where $0\leq q< m$. Then, if $q=0$,  
we have 
\begin{align} \label{Megha103} 
\ga_n& = \UU_{j=1}^{m}T_j^{-1}\{c_j+\frac{2(i-1)r\sin \frac{\pi}m}{k}  : 1\leq i\leq k+1\}  \te{ with } V_n(P)=\frac {r^2 \sin^2\frac{\pi} m}{3k^2}.
\end{align} 
On the other hand, if $0<q<m$, then there are $^mC_q$ possible sets $\ga_n$, one such set is given by 

\begin{equation} \label{Megha104} 
\begin{aligned}  \ga_n & = \Big(\UU_{j=1}^{q}T_j^{-1}\{c_j+\frac{2(i-1)r\sin \frac{\pi}m}{k+1} : 1\leq i\leq k+2\}\Big)\\
&\qquad \UU \Big(\UU_{j=q+1}^{m}T_j^{-1}\{c_j+\frac{2(i-1)r\sin \frac{\pi}m}{k}  : 1\leq i\leq k+1\}\Big) \\
\te{ with } V_n(P)&=\frac {r^2q\sin^2\frac{\pi} m}{3m(k+1)^2}+ \frac {r^2(m-q)\sin^2\frac{\pi} m}{3mk^2}.
\end{aligned} 
\end{equation} 
 \end{theorem} 
\begin{proof}
For $n=mk+1+q$, where $0\leq q<m$, let $\gg_n$ be a conditional optimal set of $n$-points for $Q$ as given by Proposition~\ref{prop45} with respect to the conditional set $\set{c_j : 1\leq j\leq m+1}$ for $n\geq m+1$. First assume that $q=0$. Then, $\gg_n$ is given by \eqref{Megha101}. Then, by Remark~\ref{rem0}, the set $\ga_n$ given by \eqref{Megha103} forms a conditional optimal set of $n$-points  with respect to the conditional set $\set{A_j : 1\leq j\leq m}$.  Next, assume that $0<q<m$. Then, $\gg_n$ is given by \eqref{Megha102}. Then, by Remark~\ref{rem0}, the set $\ga_n$ given by \eqref{Megha104} forms a conditional optimal set of $n$-points  with respect to the conditional set $\set{A_j : 1\leq j\leq m}$. If $n=mk+q$, the optimal set $\gg_n$ can be constructed in $^mC_q$ ways, and so is the set $\ga_n$.
Recall that the bijective functions $T_j$ preserve the distance as well as the collinearity of the elements in each interval $[c_j, c_{j+1}]$. Hence, the $(mk+q)$th-conditional quantization error with respect to the uniform distribution $P$ remains same as the $(mk+1+q)$th-conditional quantization error with respect to the uniform distribution $Q$. Thus, the expressions for quantization errors $V_n(P)$ given by \eqref{Megha103} and \eqref{Megha104} are followed from the expressions given by \eqref{Megha101} and \eqref{Megha102}. 
\end{proof}

\begin{theorem}\label{theo45}
Let $P$ be the uniform distribution defined on the boundary of a regular $m$-sided polygon inscribed in a circle of radius $r$ with center $(0, 0)$. Then, the conditional quantization coefficient for $P$ exists as a finite positive number and equals $\frac{1}{3} m^2r^2 \sin ^2(\frac{\pi }{m})$, i.e.,
$\lim\limits_{n\to \infty} n^2 (V_n(P)-V_\infty(P))=\frac{1}{3} m^2 r^2 \sin ^2\frac{\pi }{m}.$
\end{theorem}

\begin{proof} Let $n\in \D N$ be such that $n\geq m$. Then, there exists a unique positive integer $k$ such that $n=mk+q$ for some $0\leq q<m$. Then, by Theorem~\ref{theMe0}, we have 
\begin{align*}  V_n(P)&=\frac {r^2q\sin^2\frac{\pi} m}{3m(k+1)^2}+ \frac {r^2(m-q)\sin^2\frac{\pi} m}{3mk^2}\\
&=\frac{m^2 r^2 \sin ^2\left(\frac{\pi }{m}\right) \left(m^2+2 m n-3 m q+n^2-4 n q+3 q^2\right)}{3 (n-q)^2 (m+n-q)^2}.
\end{align*} 
Then, we see that $V_{\infty}(P)=\lim_{n\to \infty} V_n(P)=0$. In fact, we have 
\[\lim\limits_{n\to \infty} n^2 (V_n(P)-V_\infty(P))=\frac{1}{3} m^2 r^2 \sin ^2\frac{\pi }{m}.\]
\end{proof}

\begin{remark} \label{rem8} 
By Theorem~\ref{TheoM1} and Theorem~\ref{theo45}, we can say that the quantization coefficient for the uniform distribution $P$ defined on the boundary of a regular $m$-sided polygon is $\frac{1}{3} m^2r^2 \sin ^2\frac{\pi }{m}$, which is a finite positive number, but it is not a constant as it depends on both $m$ and $r$, where $m$ is the number of sides of the polygon and $r$ is the radius of the circle in which the polygon is inscribed, and this lead us to conclude that the quantization coefficient for a uniform distribution defined on the boundary of a regular $m$-sided polygon depends on both the number of sides of the polygon and the length of the sides.  
\end{remark} 

\section{Conclusion and Future Work}

In this paper, we studied the conditional quantization theory for uniform distributions supported on various geometric structures, specifically line segments, circles, and boundaries of regular polygons. For each case, we computed the conditional optimal sets of $n$-points and the corresponding $n$th conditional quantization errors under both constrained and unconstrained scenarios. We established that while the quantization dimension remains invariant---coinciding with the Euclidean dimension of the underlying space---the quantization coefficient is sensitive to the geometry of the support. Notably, we showed that:
\begin{itemize}
    \item For uniform distributions on line segments, the quantization coefficient depends solely on the length of the segment and is independent of the conditional set;
    \item For distributions on circles, the quantization coefficient depends on the radius of the circle;
    \item For distributions on boundaries of regular polygons, the quantization coefficient depends on both the radius of the circumscribing circle and the number of polygon sides.
\end{itemize}

These findings affirm that the quantization coefficient, in contrast to the quantization dimension, reflects finer structural properties of the support, such as size and shape.

Building on the results of this paper, several promising directions for future research can be identified:

\begin{enumerate}
    \item \textbf{Extension to Non-Uniform Distributions:} Investigating conditional quantization for absolutely continuous, non-uniform distributions (e.g., exponential or beta distributions) on bounded geometric supports remains an open and compelling problem.

    \item \textbf{Quantization on Fractal Supports:} Analyzing conditional quantization on self-similar and self-affine fractal sets (e.g., Cantor sets, Koch curves, Sierpiński gaskets) could further enrich the theory and reveal new structural dependencies.

    \item \textbf{Algorithmic and Computational Aspects:} Developing efficient algorithms to numerically compute optimal sets of $n$-points and corresponding quantization coefficients for complex supports and arbitrary distributions could bridge theory with practical applications.

    \item \textbf{Applications to Information Theory and Signal Processing:} Exploring how conditional and constrained quantization strategies can be leveraged in source coding, image compression, and sensor network optimization may lead to impactful applications in engineering.
\end{enumerate}

\end{document}